\documentclass[a4paper,12pt]{article}
\usepackage[T1,T2A]{fontenc}
\usepackage[utf8]{inputenc}
\usepackage[english]{babel}
\usepackage[obeyspaces]{url}
\usepackage[dvipsnames]{xcolor}
\usepackage{tabularx}

\usepackage{mathtools, amsmath, amsthm, amssymb, mathtools,
  tikz,tikz-cd, subcaption, framed, float, tensor, bm, dsfont,
  microtype, fullpage, diagbox,  makecell, mathabx, caption, subcaption}

\usepackage[normalem]{ulem}
\usepackage{thm-restate}
\usepackage{cite}

\usepackage[
  bookmarks=false, 
  colorlinks,
  citecolor=blue,
  linkcolor=blue,
  urlcolor=blue,
]{hyperref}
\usepackage{authblk}
\usepackage[
    capitalize,
    nameinlink,
    noabbrev
]{cleveref}
\usetikzlibrary{shapes.geometric, backgrounds, calc,
  arrows, automata, fit, positioning, patterns,
  decorations.pathreplacing, tikzmark, arrows.meta}

\newtheorem{thm}{Theorem}[section]
\newtheorem{lemma}[thm]{Lemma}
\newtheorem{cor}[thm]{Corollary}

\newtheorem{obs}[thm]{Observation}

\theoremstyle{definition}
\newtheorem{definition}[thm]{Definition}
\newtheorem{remark}[thm]{Remark}

\newcommand\oeis[1]{\href{https://oeis.org/#1}{#1}}

\newcommand{\N}{\mathbb{N}}

\newcommand{\letw}{\operatorname{\mathtt{letw}}}
\newcommand{\retw}{\operatorname{\mathtt{retw}}}

\let\ge\geqslant
\let\leq\leqslant
\let\geq\geqslant

\tikzstyle{dot}=[fill, circle, minimum size=0.5ex, inner sep=0pt]
\tikzstyle{ml}=[line width=0.7pt, -]
\tikzstyle{mb}=[line width=1pt, -]
\def\ptwoone{%
  \tikz[scale=0.5]{
    \node[dot] (a) at(0,0) {};
    \node[dot] (b) at(0.5,0) {};
    \node[dot] (c) at(1,0) {};
    \node[dot] (d) at(1.5,0) {};
    \draw[out=70,in=110, ml] (a) edge (d);
    \draw[out=70,in=100, ml] (b) edge (c);
    \draw[ml] (a) edge (b);
    \draw[ml] (c) edge (d);
  }%
}
\def\ponetwo{%
  \tikz[scale=0.5]{
    \node[dot] (a) at(0,0) {};
    \node[dot] (b) at(0.5,0) {};
    \node[dot] (c) at(1,0) {};
    \node[dot] (d) at(1.5,0) {};
    \draw[out=70,in=110, ml] (a) edge (c);
    \draw[out=70,in=90, mb, white] (b) edge (d);
    \draw[out=70,in=90, ml] (b) edge (d);
    \draw[ml] (a) edge (b);
    \draw[ml] (c) edge (d);
  }%
}
\def\pthreetwoone{
  \tikz[scale=0.5]{
    \node[dot] (a) at(0,0) {};
    \node[dot] (b) at(0.5,0) {};
    \node[dot] (c) at(1,0) {};
    \node[dot] (d) at(1.5,0) {};
    \node[dot] (e) at(2,0) {};
    \node[dot] (f) at(2.5,0) {};
    \draw[out=60,in=120, ml] (a) edge (f);
    \draw[out=60,in=120, ml] (b) edge (e);
    \draw[out=60,in=120, ml] (c) edge (d);
    \draw[ml] (a) edge (b);
    \draw[ml] (b) edge (c);
    \draw[ml] (d) edge (e);
    \draw[ml] (e) edge (f);
  }%
}

\def\ponetwothree{
  \tikz[scale=0.5]{
    \node[dot] (a) at(0,0) {};
    \node[dot] (b) at(0.5,0) {};
    \node[dot] (c) at(1,0) {};
    \node[dot] (d) at(1.5,0) {};
    \node[dot] (e) at(2,0) {};
    \node[dot] (f) at(2.5,0) {};
    \draw[out=60,in=120, ml] (a) edge (d);
    \draw[out=60,in=120, mb, white] (b) edge (e);
    \draw[out=60,in=120, ml] (b) edge (e);
    \draw[out=60,in=120, mb, white] (c) edge (f);
    \draw[out=60,in=120, ml] (c) edge (f);
    \draw[ml] (a) edge (b);
    \draw[ml] (b) edge (c);
    \draw[ml] (d) edge (e);
    \draw[ml] (e) edge (f);
  }%
}

\def\ponethreetwo{
  \tikz[scale=0.5]{
    \node[dot] (a) at(0,0) {};
    \node[dot] (b) at(0.5,0) {};
    \node[dot] (c) at(1,0) {};
    \node[dot] (d) at(1.5,0) {};
    \node[dot] (e) at(2,0) {};
    \node[dot] (f) at(2.5,0) {};
    \draw[out=60,in=120, ml] (a) edge (d);
    \draw[out=60,in=120, mb, white] (c) edge (e);
    \draw[out=60,in=120, ml] (c) edge (e);
    \draw[out=60,in=120, mb, white] (b) edge (f);
    \draw[out=60,in=120, ml] (b) edge (f);
    \draw[ml] (a) edge (b);
    \draw[ml] (b) edge (c);
    \draw[ml] (d) edge (e);
    \draw[ml] (e) edge (f);
  }%
}
\def\ptwoonethree{
  \tikz[scale=0.5]{
    \node[dot] (a) at(0,0) {};
    \node[dot] (b) at(0.5,0) {};
    \node[dot] (c) at(1,0) {};
    \node[dot] (d) at(1.5,0) {};
    \node[dot] (e) at(2,0) {};
    \node[dot] (f) at(2.5,0) {};
    \draw[out=60,in=120, ml] (a) edge (e);
    \draw[out=60,in=120, ml] (b) edge (d);
    \draw[out=60,in=120, mb, white] (c) edge (f);
    \draw[out=60,in=120, ml] (c) edge (f);
    \draw[ml] (a) edge (b);
    \draw[ml] (b) edge (c);
    \draw[ml] (d) edge (e);
    \draw[ml] (e) edge (f);
  }%
}
\def\pthreeonetwo{
  \tikz[scale=0.5]{
    \node[dot] (a) at(0,0) {};
    \node[dot] (b) at(0.5,0) {};
    \node[dot] (c) at(1,0) {};
    \node[dot] (d) at(1.5,0) {};
    \node[dot] (e) at(2,0) {};
    \node[dot] (f) at(2.5,0) {};
    \draw[out=60,in=120, ml] (c) edge (d);
    \draw[out=60,in=120, ml] (a) edge (e);
    \draw[out=60,in=120, mb, white] (b) edge (f);
    \draw[out=60,in=120, ml] (b) edge (f);
    \draw[ml] (a) edge (b);
    \draw[ml] (b) edge (c);
    \draw[ml] (d) edge (e);
    \draw[ml] (e) edge (f);
  }%
}
\def\ptwothreeone{
  \tikz[scale=0.5]{
    \node[dot] (a) at(0,0) {};
    \node[dot] (b) at(0.5,0) {};
    \node[dot] (c) at(1,0) {};
    \node[dot] (d) at(1.5,0) {};
    \node[dot] (e) at(2,0) {};
    \node[dot] (f) at(2.5,0) {};
    \draw[out=60,in=120, ml] (c) edge (e);
    \draw[out=60,in=120, ml] (a) edge (f);
    \draw[out=60,in=120, mb, white] (b) edge (d);
    \draw[out=60,in=120, ml] (b) edge (d);
    \draw[ml] (a) edge (b);
    \draw[ml] (b) edge (c);
    \draw[ml] (d) edge (e);
    \draw[ml] (e) edge (f);
  }%
}

\let\phi\varphi

\let\epsilon\varepsilon

\DeclareUnicodeCharacter{21A6}{$\mapsto$}

\author{Célia Biane$^1$, Greg Hampikian$^2$, Sergey Kirgizov$^1$, Khaydar Nurligareev$^1$}
\affil{\small \rm $^1$LIB, Université de Bourgogne, Dijon, France \protect\\
 \rm $^2$CompGenomics, Box 1454, Boise, Idaho 83701, USA

 Corresponding author: sergey.kirgizov@u-bourgogne.fr
  }
\date{\today}

\title{Endhered patterns in matchings and RNA}

\begin{document}

\maketitle

\begin{abstract}
  An {\em endhered (end-adhered) pattern} is a subset of arcs in
  matchings, such that the corresponding starting points are
  consecutive and the same holds for the ending points. Such patterns
  are in one-to-one correspondence with the permutations. We focus on
  the occurrence frequency of such patterns in matchings and
  native (real-world) RNA structures with pseudoknots.  We present
  combinatorial results related to the distribution and asymptotic behavior
  of the pattern 21, which corresponds to two consecutive base pairs
  frequently encountered in RNA, and the pattern 12,
  representing the archetypal minimal pseudoknot.  We show that in
  matchings these two patterns are equidistributed, which is quite
  different from what we can find in native RNAs. We also examine the
  distribution of endhered patterns of size 3, showing how the
  patterns change under the transformation called {\em endhered
    twist}. Finally, we compute the distributions of endhered patterns
  of size 2 and 3 in native secondary RNA structures with
  pseudoknots and discuss possible outcomes of our study.
\end{abstract}

\section{Introduction}

Ribonucleic acids (RNAs) are macromolecules fulfilling many
biological functions: they code for protein, are involved in the
regulation of gene expression, can have catalytic activities and store
the genetic information of certain viruses. The structure of RNAs is
defined at the primary level as sequences of four nucleotides: Adenine
(A), Uracil (U), Guanine (G), and Cytosine (C). The secondary structure
abstracts from the nature of the nucleotide and considers only the
bonds forming between nucleotides during the synthesis of RNAs and
shaping how the molecules folds in space.
Two types of bonds are formed during the RNA folding process:
phospho-diester bonds (known as strong bonds) are formed between pairs
of consecutive nucleotides in the sequence forming the RNA chain, and
hydrogen bonds (also known as weak bonds) are formed between pairs of
nucleotides distant in the sequence.
The secondary structure represents an intermediate level between the
primary sequence and the shape, and has the advantage of being both
relevant from a biological perspective and tractable from a
computational point of view.

\subsection{Models of RNA secondary structures}
\label{subsec:models}

RNA secondary structures have been formalized as graphs primarily by
Waterman~\cite{waterman} to tackle the problem of prediction of
secondary structures from the primary structure, which is more easily
measurable. Ponty~\cite{ponty} gives a variant of Waterman definition
of RNA secondary structure without pseudoknots and having a minimal
number $\theta$ of unpaired positions between pair positions.
Formally, {\em Waterman-Ponty RNA secondary structure} $S$ of size $n$
is defined as a set of base-pairs $(i,j), 1 \leq i < j \leq n$, such that:
\begin{enumerate}
    \item Each position is monogamous, $\forall (i,j) \ne (i',j') \in S : \{i,j\} \cap \{i',j'\} = \emptyset $.
    \item Minimal distance $\theta$ between paired nucleotides, $\forall(i,j) \in S :j - i \geq \theta$.
    \item No pseudoknot allowed, $\forall (i,j),(i',j') \in S, i < i':(j'<j)$ or $(j<i')$.
\end{enumerate}

These structures can be represented using the {\em dot-bracket
  notation}.  A secondary structure of an $n$-nucleotide RNA is
encoded as an $n$-length sequence of parentheses \{``('',``)''\} and dots
``.'', where an open parenthesis represents a nucleotide paired to
another nucleotide represented by a closed parenthesis, and dots
correspond to unpaired nucleotides.  {\em The extended dot-bracket
  notation} includes also other types of parentheses: ``[]'',
``\{\}'', ``<>'', ``aA'', etc.  The extended dot-bracket notation
enables the representation of pairings in
pseudoknots. Figure~\ref{fig:enter-label} shows 4 different
representations of an example of RNA secondary structure.

\begin{figure}[h]
  \centering
  \begin{subfigure}[t]{0.3\textwidth}
    \centering
    \includegraphics[width=10em]{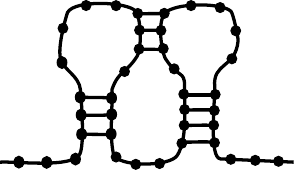}
    \caption{}
    \label{fig:enter-label:A}
  \end{subfigure}
  \begin{subfigure}[t]{0.6\textwidth}
    \centering
    \includegraphics[width=20em]{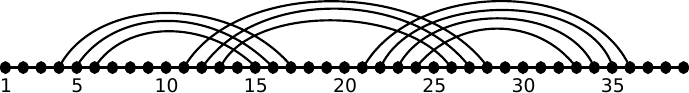}
    \caption{}
    \label{fig:enter-label:B}
  \end{subfigure}

  \vspace{1em}
  
  \begin{subfigure}[t]{0.3\textwidth}
    \centering
    \includegraphics[width=10em]{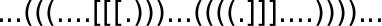}
    \caption{}
    \label{fig:enter-label:C}
  \end{subfigure}
  \begin{subfigure}[t]{0.6\textwidth}
    \centering
    \includegraphics[width=20em]{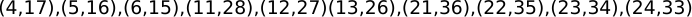}
    \caption{}
    \label{fig:enter-label:D}
  \end{subfigure}
  \caption{A drawing (a), an arc diagram (b),
    an extended dot bracket notation (c), and a set of pairs (d)
    representing an example of an RNA secondary
    structure.}
  \label{fig:enter-label}
\end{figure}

Other models of RNA secondary structures were studied from
combinatorial point of view. Haslinger and Standler~\cite{haslinger}
examined enumerative and asymptotic properties of {\em bi-secondary
  structures,} i.e., arc diagrams with arcs both in upper
  and lower part of the plane but without arc intersections.  The
concept of RNA shape was introduced by Giegerich, Voss, and
Rehmsmeier~\cite{giegerich} who developed an algorithm for the
computation of minimum free energy RNA shapes. The number and
asymptotics of RNA shapes was studied by Lorenz, Ponty, and
Clote~\cite{lorenz2008}.  Reidys and Wang~\cite{reidys}
  considered a generalization of RNA shapes defined on $k$
  mutually-crossing arcs.

\medskip
  
In this paper, we look at RNA structures with no restrictions on the number of arc crossings.

\subsection{Notion of pattern in RNA secondary structures}

At the biological level, common patterns have been observed in orthodox
structures: single strand regions, hairpin and internal loops, bulges
and various computational tools exist for detecting these patterns from primary
sequence, including the work of Macke {\em et al.}~\cite{macke2001}.
More and more information is being gathered from three dimensional
reconstructions of RNA molecules paving the way to a better
comprehension of the laws governing the RNA folding process and the
formation of RNA patterns.

Rødland~\cite{rodland} proposed a classification of RNA secondary
structures in four level of abstractions: nucleotide, ladder, stem and
collapsed level, based on the considered internal patterns.  In his
work, the nucleotide level corresponded to structures with arc
diagrams containing unpaired nucleotides, the ladder level
corresponded to structures abstracted from unpaired nucleotides, the
stem level was abstracted from bulges and internal loops, and the
collapsed level was abstracted from nested loops. Rødland studied
different kind of pseudoknot patterns of increasing complexity:
H-pseudoknot, double hairpin pseudoknot and pseudotrefoil. He counted
these patterns in RNA secondary structures of increasing complexity
and studied their asymptotics. He also showed that the theoretical
number of pseudoknots in secondary structures is higher than in real
secondary structure of the Rfam~\cite{rfam2003, rfam2005}
and PseudoBase~\cite{pseudobase} databases. Note that Rødland's
collapsed structures correspond to RNA shapes studied by Giegerich,
Voss, and Rehmsmeier~\cite{giegerich}.

Quadrini~\cite{quadrini2021} addressed the problem of searching a
given structural pattern, defined as a sequence of crossing loops in a
RNA secondary structure or shape and characterized by arbitrary number
of pseudoknots. She proposed polynomial time algorithms for their
identification. A paper by Gan, Pasquali, and Schlick~\cite{gan}
studied RNA structures and their patterns using graph-based
representations. In the future, it would be interesting to compare the
relationships between different kinds of patterns.

\bigskip

In this paper, we study the formation of endhered patterns (formally
defined below) in matchings.  In Section~\ref{section:t}, we do it
from a theoretical perspective.  In Section~\ref{section:p}, we
observe the number of occurrences of such patterns in RNA secondary
structures derived from experimentally determined 3D RNA structures.
We conclude by discussing possible outcomes of our study in
Section~\ref{section:c}.

\section{Endhered patterns in matchings}
\label{section:t}

\subsection{Basic definitions}
\label{subsection:def}

By (perfect) {\em matching} of size $n$, we mean the sequence of $2n$
points $(1, 2, \ldots, 2n-1, 2n)$ endowed with a set of $n$ arcs, such
that every point is linked to one and only one another point.
Figure~\ref{fig:example:A} shows examples of matchings of small sizes. 
Matchings of size $n$ can be considered as
fixed-point-free involutions in the symmetric group $S_{2n}$.  Thus,
the matching with 4 arcs at the bottom of Figure~\ref{fig:example:A}
can be represented by the
permutation $3\,6\,1\,5\,4\,2\,8\,7$, which is the product of disjoint
transpositions $(13)(26)(45)(78)$.

\begin{figure}[h]
  \centering
  \begin{subfigure}[t]{0.45\textwidth}
    \includegraphics{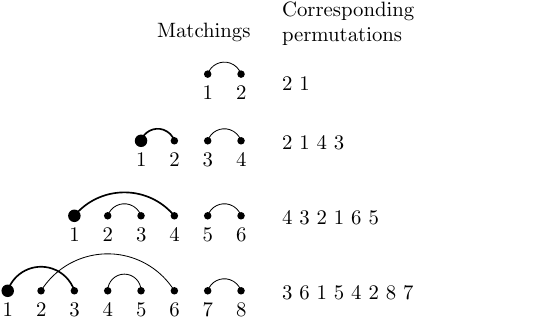}
    \caption{}
    \label{fig:example:A}
  \end{subfigure}\quad
  \begin{subfigure}[t]{0.45\textwidth}
    \includegraphics{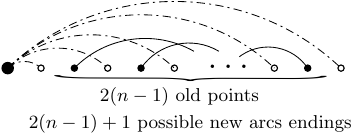}
    \caption{}
    \label{fig:example:B}
  \end{subfigure}
  \caption{An example of matching construction, corresponding permutations (a),
    and a schema of recursive construction of matchings (b).}
  \label{fig:example}
\end{figure}

 For a positive integer $n$, any matching of size $n$
  can be uniquely constructed from some matching of size $n-1$ by the
  following procedure.  We add a new arc starting at the left of the
  already existing $2(n-1)$ points and ending at some of $2(n-1)+1$
  possible new positions (see Figure~\ref{fig:example:B}).
  This observation leads us to the
  following recurrence relation for the number of matchings of size
  $n$: $a_n = (2n-1)a_{n-1}$ with $a_0 = 1$.  As a consequence, $a_n =
  (2n-1)!! = (2n-1)\cdot(2n-3)\cdot\ldots\cdot3\cdot1$.  The
  corresponding sequence starting with $1, 1, 3, 15, 105, 945, 10395$
  is known as \oeis{A1147} in Sloane's Encyclopedia~\cite{oeis}.

In mathematical literature, matchings often appear in different
contexts, from the representation theory of Lie algebras~\cite{csm} to
the geometry of moduli spaces of flat connections on
surfaces~\cite{amr}.  Efficient generating algorithms for involutions
with (without) fixed points were established by Vajnovszki~\cite{vaj}.
Ardnt's book~\cite{matters-computational} also presents interesting
generating algorithms for involutions and other combinatorial
structures.

  Different kinds of patterns in matchings are being actively studied
  in combinatorics for the past twenty years.
  Initially, the interest to this topic came from the rapidly
  developing study of permutation patterns, since a matching can be
  thought as a permutation of a specific form.  From this perspective,
  we say that a matching $\sigma$ is a {\em pattern} in a matching
  $\mu$ if $\sigma$ can be obtained from $\mu$ by deleting some of its
  arcs (and consistently relabelling the remaining vertices).  For
  instance, Chen, Deng, Du, Stanley, and Yan~\cite{cddsy} studied
  distributions of crossings and nestings.
  Jelínek~\cite{jel}, as well as Bloom and Elizalde~\cite{be},
  considered pattern avoiding matchings in the case when $\sigma$ is a
  permutational matching of size 3.  An extension for more general
  patterns was elaborated by Cervetti and Ferrari~\cite{matteo-luca},
  while other authors, such as Chen, Mansour and Yan~\cite{chmy},
  Jelínek and Mansour~\cite{jm}, considered partial patterns.

As we see, what matters in the above investigations is the relative
positions of arcs that form a pattern.  At the same time, the
distances between starting and ending points of these arcs are not
fixed.  The main object of our study concerns specific restrictions
imposed on the arcs.  Namely, the starting points of a pattern, as
well as its ending points, form an interval, while the distance
between these two intervals may vary.  We call such patterns {\em
  endhered} (end-adhered) to emphasize the nature of these
restrictions. 

\bigskip

\begin{definition}
  An {\em endhered pattern} is a matching, such that the starting
  point of any of its arcs precedes the ending point of any other arc.
  In other words, a matching
  of size $p$ written as a
  permutation $\sigma = \sigma_1\ldots\sigma_{2p}$ is an endhered
  pattern if $\pi = \sigma_{p+1}\ldots\sigma_{2p}$ is a permutation of size
  $p$ (such matchings are also called {\em permutational}).
  Figure~\ref{fig:definition:A} presents an example of an endhered
  pattern of size $3$.  In the following, we identify endhered
  patterns with the corresponding permutations.

  We say that a matching $\mu = \mu_1\ldots\mu_{2n}$ {\em contains} an
  endhered pattern $\pi = \pi_1\ldots\pi_{p}$ at position $(i+1,j+1)$,
  where $i\geqslant0$ and $i+p \leqslant j \leqslant n-p$, if
  \[
    \mu_{s+i} = \pi_s^{-1} + j,
    \qquad
    s = 1, \ldots, p
  \]
  (here, $\pi^{-1} = \pi_1^{-1}\ldots\pi_{p}^{-1}$ is the inverse to
  the permutation $\pi$).  In other words, $\mu$ contains $p$ arcs
  whose starting points are $i+1, \ldots, i+p$, whose ending points
  are $j+1, \ldots, j+p$, and that form the endhered pattern $\pi$.
  Figure~\ref{fig:definition:B} shows an example
  of a matching containing pattern shown on Figure~\ref{fig:definition:A}.
\label{def:endhered}
\end{definition}

\begin{figure}[h]
  \centering
  \begin{subfigure}[t]{0.3\textwidth}
      \centering
    \includegraphics[]{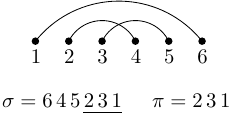}
    \caption{}
    \label{fig:definition:A}
  \end{subfigure}\hspace{0.3em}
  \begin{subfigure}[t]{0.6\textwidth}
    \centering
    \includegraphics[]{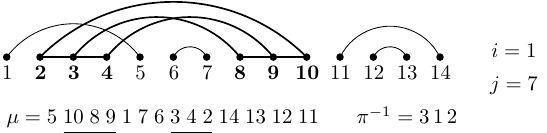}
    \caption{}
    \label{fig:definition:B}
  \end{subfigure}
  
  \caption{Endhered pattern 231 (a) and an example (b) of its occurrence.}
  \label{fig:definition}
\end{figure}

The endhered patterns have been rarely studied in the literature,
although they may shed light on the formation of collapsed RNA
structures from Rødland's paper about pseudoknots~\cite{rodland}.  The
only work in this direction that we are aware of is the paper of
Baril~\cite{baril} who examined one of two endhered patterns of size
$2$ in his study on irreducible involutions and permutations.

\subsection{Endhered twists and engendered symmetries}
\label{subsec:twist}

Given an endhered pattern $\pi$, let us denote by $a_{n,k}(\pi)$ the
number of matchings of size $n$ with $k$ occurrences of $\pi$.  If,
additionally, $\tau$ is another endhered pattern, then we designate
by $a_{n,k,m}(\pi,\tau)$ the number of matchings of size $n$ with
$k$ and $m$ occurrences of patterns $\pi$ and $\tau$, respectively.
Certain patterns, for instance $\pi=\ptwoone$ and $\tau=\ponetwo$,
have the same distribution, meaning that
$a_{n,k}(\pi)=a_{n,k}(\tau)$.  The goal of this Subsection is to
establish such equidistributed classes of endhered patterns with the
help of bijections, i.e., without direct enumeration.  To this end, we
apply matching transformations that we call {\em endhered twists}.

\begin{definition}
  The {\em left endhered twist} (resp. {\em right endhered twist}) is
  a transformation that takes a matching $\mu$ and produces a matching $\letw(\mu)$ (resp. $\retw(\mu)$)
  such that all runs of consecutive starting (resp. ending) points are
  reversed.
  Figure~\ref{fig:endhered-twist} shows an example of the right twist.
  \label{def:endhered-twist}
\end{definition}

\begin{figure}[h]
  \centering
  \includegraphics{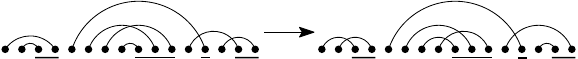}
  \caption{An example of the right endhered twist, runs of right points are underlined.}
  \label{fig:endhered-twist}
\end{figure}

\begin{obs}
  Endhered twists of endhered patterns correspond to classical symmetries on permutations.
  Thus, the right endhered twist applied to an endhered pattern $\pi=\pi_1\ldots\pi_p$ is its reverse: $\retw(\pi)=\pi_p\ldots\pi_1$.
  At the same time, the left endhered twist is the complement:  $\letw(\pi)=(p+1-\pi_1)\ldots(p+1-\pi_p)$.
  For example,
  $$
  \begin{aligned}
    \retw(321)  = 123, \qquad & \retw(231) = 132, \qquad \retw(312) = 213, \\
    \letw(321)  = 123, \qquad & \letw(132) = 312, \qquad \letw(213) = 231.\\
  \end{aligned}
  $$
  Geometrically, if we represent permutations as square tables, the endhered twists are axial symmetries (see Figure~\ref{fig:retw}).
  \label{obs:endhered-twist}
\end{obs}

\begin{figure}
  \centering
  \includegraphics[]{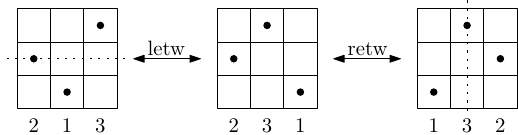}
  \caption{Geometrical meaning of endhered twists.}
  \label{fig:retw}
\end{figure}

\begin{lemma}
  Two endhered patterns $\pi$ and $\tau$ of the same size have the same joint distribution
  if they are identical under the left or right endhered twists.
  In other words, if $\pi = \letw(\tau)$ or $\pi = \retw(\tau)$, then
  \begin{equation}
    a_{n,k,m}(\pi,\tau) = a_{n,k,m}(\tau,\pi)
   \label{eq:etwist}
  \end{equation}
  for any integers $n$, $k$ and $m$.
  \label{lemma:etwist}
  In particular, $a_{n,k}(\pi) = a_{n,k}(\tau)$.
\end{lemma}
\begin{proof}
  Suppose that $\pi$ is obtained from $\tau$ by applying the left endhered twist.
  Let $\mu$ be a matching containing $k$ occurrences of $\pi$ and $m$ occurrences of $\tau$.
  Applying the left twist to the whole $\mu$, we transform every occurrence of $\pi$ to $\tau$ and vice versa.
  No other occurrence of $\pi$ or $\tau$ are created.
  This implies relation~\eqref{eq:etwist}.
  The cases of the right twist is obtained {\em mutatis mutandis}.
\end{proof}

\begin{remark}
  Given a positive integer $n$ and an endhered pattern $\pi$, the
  value $a_{n,0}(\pi)$ is the number of matchings of size $n$ avoiding
  $\pi$.  We say that two endhered patterns $\pi$ and $\tau$ are
  {\em Wilf equivalent} if $a_{n,0}(\pi) = a_{n,0}(\tau)$ for every
  $n$.  For example, Wilf equivalent patterns of size 2 and 3 are
  shown in Figure~\ref{fig:wilf}.  Matchings avoiding the pattern
  \ptwoone\ correspond to involutions with no fixed points and no
  successions considered by Baril~\cite{baril}.

  For the endhered patterns of size 2 and 3, as we will see in this
  paper, their Wilf equivalence implies that the corresponding
  patterns are equidistributed.  In other words, if $a_{n,0}(\pi) =
  a_{n,0}(\tau)$ for every $n$, then $a_{n,k}(\pi) =
  a_{n,k}(\tau)$ for all $n$ and $k$.  In the general case, this
  question is not trivial.
\end{remark}

\begin{figure}[h]
  \centering
  \includegraphics[width=0.8\textwidth]{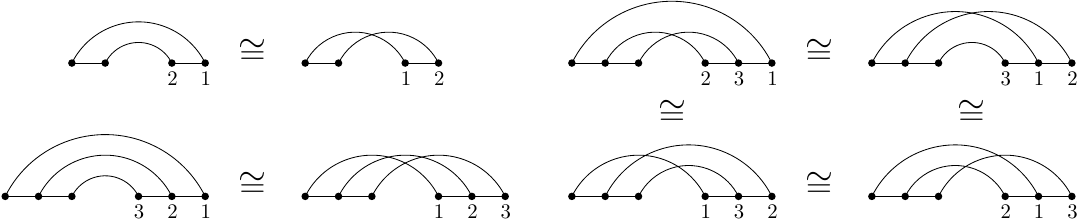}
  \caption{Wilf equivalence classes of endhered patterns of size 2 and 3.}
  \label{fig:wilf}
\end{figure}

\begin{cor}
  The joint distribution of the endhered patterns $\pi=1\ldots p$ and
  $\tau = p\ldots 1$ is symmetric.  In particular, the number of
  size-$n$ matchings containing $k$ \ptwoone\ and $m$ \ponetwo\ equals
  the number of size-$n$ matchings containing $m$ \ptwoone\ and $k$
  \ponetwo.
  \label{cor:12_21}
\end{cor}

\begin{remark}
  Corollary~\ref{cor:12_21} is consistent with the result of Chen,
  Deng, Du, Stanley, and Yan~\cite{cddsy} who showed that, in the case
  of classical patterns, the joint distribution of crossings and
  nestings is symmetric.
\end{remark}

\subsection{Enumeration of endhered patterns of size 2}
\label{subsection:patterns2}

The goal of this Subsection is to study the distributions of endhered patterns of size 2.
It follows from Lemma~\ref{lemma:etwist} that patterns \ptwoone\ and \ponetwo\ have the same distribution, so we need only to enumerate one of them.
For simplicity, we will write $a_{n,k}$ instead of $a_{n,k}(21)$ throughout the rest of the paper.
Note that some of results presented here (Lemma~\ref{l12}, Corollary~\ref{cor_zeroth_row}) were proved by Baril~\cite{baril}
using the language of permutations and involutions.
For the sake of completeness, we provide new versions of these
proofs using the language of matchings.

\begin{thm}
  For $n\ge1$, the number of size-$n$ matchings containing exactly $k$ occurrences  of pattern \ptwoone\ (resp. \ponetwo) satisfies
  \begin{equation}
    \left\{
    \begin{array}{lll}
      a_{n+1,k} & = & a_{n,k-1} + 2(n-k) \cdot a_{n,k} + 2(k+1) \cdot   a_{n,k+1}\\
      a_{1,0} & = & 1\\
      a_{1,k} & = & 0, \quad k \ne 0.
    \end{array}
    \right.
    \label{eq:12}
  \end{equation}
  Formally, we allow negative values of $k$. It follows from~\eqref{eq:12} that $a_{n,k} = 0$ whenever $k<0$.
  \label{theorem:12}
\end{thm}

\begin{proof}
  The boundary values corresponding to $n=1$ are easily obtained.
  To construct a matching of size $n+1$ having $k$ occurrence of
  \ptwoone\,, we take a non-empty matching $\mu$ of size $n$ and add a
  new arc $E$ together with its two ending points (left and right).
  We always put the left point of $E$ before the first point of $\mu$.
  For the right point of $E$, there are three different cases, since
  the new arc either creates a new pattern, destroys an existing
  pattern, or leaves the number of occurrences of the pattern unchanged.
  \begin{enumerate}
  \item 
    The right point of $E$ is placed just after the right point of the arc $(1,i)$ of $\mu$ (see Figure~\ref{fig:thm1fig1}).
    In this case, the number of occurrences increases by one.
    Hence, the matching $\mu$ must possess $k-1$ patterns \ptwoone\,, so that the resulting matching has $k$ occurrences of this pattern.
    This gives us the first term in relation~\eqref{eq:12}.
    \begin{figure}[h]
      \centering
      \includegraphics[width=0.4\textwidth]{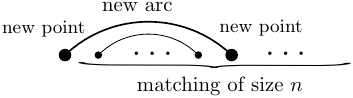}
      \caption{Adding a new arc creates a new pattern.}
      \label{fig:thm1fig1}
    \end{figure}

  \item 
    The right end of $E$ is put between the left (resp. right) ends of any two arcs that make up the pattern \ptwoone\,.
    In this case, one of the existing patterns is destroyed, and the number of occurrences decreases by one (see  Figure~\ref{fig:thm1fig3}).
    Therefore, to have $k$ occurrences in the resulting matching, the initial matching $\mu$ must contain $k+1$ occurrences of \ptwoone\,.
    Since there are $2(k+1)$ destructive positions, this gives us the last term in relation~\eqref{eq:12}.
    \begin{figure}[h]
      \centering
      \includegraphics[width=0.6\textwidth]{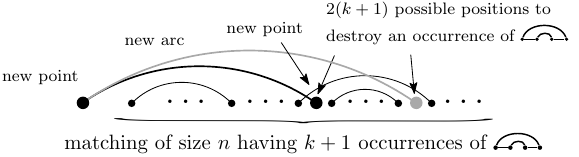}
      \caption{Adding a new arc destroys an existing pattern.}
      \label{fig:thm1fig3}
    \end{figure}

  \item
    The right point of $E$ does not create, nor remove any pattern \ptwoone\,. 
    In this case, the matching $\mu$ must possess $k$ occurrences of this pattern.
    There are $2n+1$ possible positions for the right end of $E$, but $2k+1$ of them are forbidden.
    Indeed, there are $k$ already existing occurrence of \ptwoone\ that we do not want to destroy,
    and one more position is banned because we are not allowed to create any new \ptwoone\ in this case (see  Figure~\ref{fig:thm1fig2}).
    \begin{figure}[h]
      \centering
      \includegraphics[width=0.6\textwidth]{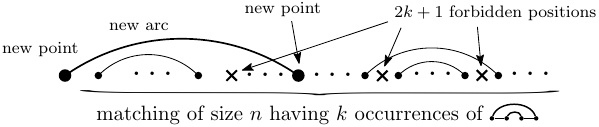}
      \caption{Adding a new arc does not change the number of pattern occurrences.}
      \label{fig:thm1fig2}
    \end{figure}
    
  \end{enumerate}
  The distribution of the pattern \ponetwo\ respects the same recurrence. To establish the latter fact, it is sufficient to
  twist the corresponding ends using Lemma~\ref{lemma:etwist}.
\end{proof}

Table~\ref{table:12} shows first values of $a_{n,k}$.
With this table and those that follow in the text, we can see the
big picture of the initial terms, intuitively understand how they
grow, formulate hypotheses about the global (asymptotic) behaviour
of these numerical sequences.  We present our results in this
direction.

\begin{table}[ht]
  \small
  \centering
  \begin{tabular}{r|ccccccccc|l}
\diagbox[width=2.5em,height=2.5em]{k}{n} &  1 & 2 & 3 & 4 & 5 & 6 & 7 & 8 & 9  & OEIS \\ \hline
0 & 1 & 2 & 10 & 68 & 604 & 6584 & 85048 & 1269680 & 21505552 & \oeis{A165968} \\
1 & 0 & 1 & 4 & 30 & 272 & 3020 & 39504 & 595336 & 10157440 &  \oeis{A179540} \\
2 & 0 & 0 & 1 & 6 & 60 & 680 & 9060 & 138264 & 2381344 & \\
3 & 0 & 0 & 0 & 1 & 8 & 100 & 1360 & 21140 & 368704 & \\
4 & 0 & 0 & 0 & 0 & 1 & 10 & 150 & 2380 & 42280 & \\
5 & 0 & 0 & 0 & 0 & 0 & 1 & 12 & 210 & 3808 & \\
6 & 0 & 0 & 0 & 0 & 0 & 0 & 1 & 14 & 280 & \\
7 & 0 & 0 & 0 & 0 & 0 & 0 & 0 & 1 & 16 & \\
8 & 0 & 0 & 0 & 0 & 0 & 0 & 0 & 0 & 1 & \\
\end{tabular}
  \caption{Distribution of pattern \protect\ptwoone\ (21).}
\label{table:12}
\end{table}

\bigskip

\begin{lemma}
  For any $n>k\geq0$,
  \begin{equation}
    a_{n,k} = \binom{n-1}{k}a_{n-k,0}.
    \label{eq:l12}
  \end{equation}
  \label{l12}
\end{lemma}


\begin{proof}
  Each matching of size $n$ with $k$ occurrences of \ptwoone\ can be uniquely obtained by the following procedure. 
  Take a matching of size $n-k$ avoiding \ptwoone\ pattern.
  Add $k$ new arcs expanding some of the original arcs into double, triple, $\ldots$, and $(k+1)$-uple arcs.
  Since each $t$-uple arc contains $t-1$ occurrences of \ptwoone\,, we have $k$ occurrences of \ptwoone\ in total (see
  Figure~\ref{fig.expand1}).
  Therefore, relation~\eqref{eq:l12} comes from the classical stars-and-bars argument~\cite{feller}:
  the binomial coefficient $\binom{n-1}{k}$ is the number of ways to add $k$ new arcs by expanding some of the $n-k$ original arcs.
  
  \begin{figure}[ht]
    \centering
    \includegraphics[width=0.6\textwidth]{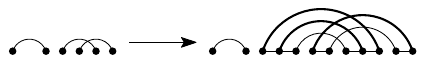}
    \caption{Arc expansion. Adding 3 new arcs, we create 3 new occurrences of \protect\ptwoone\,.}
    \label{fig.expand1}
  \end{figure}

\end{proof}

\begin{cor}
  The sequence $(a_{n,0})$ satisfies the recurrence relation
  \begin{equation}
    a_{n+1,0} = 2na_{n,0} + 2(n-1)a_{n-1,0}
    \label{eq_zeroth_row}
  \end{equation}
  with the initial conditions
  $$
    a_{1,0} = 1
    \qquad\mbox{and}\qquad
    a_{2,0} = 2.
  $$
  In particular, this sequence corresponds to a shift of \oeis{A165968} in OEIS~\cite{oeis}.
  \label{cor_zeroth_row}
\end{cor}
\begin{proof}
  Recurrence~\eqref{eq_zeroth_row} follows immediately from relations~\eqref{eq:12} and~\eqref{eq:l12}.
\end{proof}

\begin{remark}
  There is a constructive combinatorial explanation of
  relation~\eqref{eq_zeroth_row} (see also the work of
  Baril~\cite{baril}).  Indeed, if we remove the first arc in a
  matching of size $n+1$ avoiding \ptwoone\,, then we obtain a matching
  of size $n$ that either avoids \ptwoone\ or contains exactly one
  pattern \ptwoone\,.  This observation shows that any matching of size
  $n+1$ avoiding \ptwoone\ can be uniquely constructed in the
  following way.  Either we take a matching $\mu$ of size $n$ avoiding
  \ptwoone\ and add a new arc starting at the left of the already
  existing $2n$ points of $\mu$ and ending at some of $2n$ possible
  new positions (it cannot arrive just after the right point of the
  first arc of $\mu$).  Or we take a matching of size $n-1$
  avoiding \ptwoone\,, double one of its $n-1$ arcs, and then destroy
  the just created pattern \ptwoone\ by a new arc drawn in one of two
  possible ways (see Figure~\ref{fig:21-rec}).

  \begin{figure}[h]
  \begin{center}
    \includegraphics{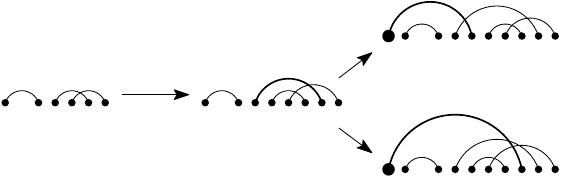}
  \end{center}
  \caption{Two possible ways for destroying the newly created
    pattern \protect\ptwoone.}
  \label{fig:21-rec}
  \end{figure}
    
  \label{rem_zeroth_row}
\end{remark}

\begin{cor}
  The sequence $(a_{n,1})$ satisfies the recurrence relation
  \begin{equation}
    a_{n+1,1} = 2n(a_{n,1} + a_{n-1,1})
    \label{eq_first_row}
  \end{equation}
  with the initial conditions
  $$
    a_{1,1} = 0
    \qquad\mbox{and}\qquad
    a_{2,1} = 1.
  $$
  In particular, this sequence corresponds to a shift of \oeis{A179540} in OEIS~\cite{oeis}.
  \label{cor_first_row}
\end{cor}
\begin{proof}
  Apply relations~\eqref{eq:12} and~\eqref{eq:l12}:
  $$
    a_{n+1,1} = na_{n,0} =
    n\Big(2(n-1)a_{n-1,0} + 2a_{n-1,1}\Big) =
    2n(a_{n,1} + a_{n-1,1}).
  $$
\end{proof}

\begin{lemma}
  For any $n\geq0$,
  \begin{equation}
    a_{n+1,0} = \sum\limits_{k=0}^n (-1)^{n-k}\binom{n}{k}(2k+1)!!,
    \label{eq_zeroth_row_2}
  \end{equation}
  where $m!!$ denotes the {\em double factorial} of $m$,
  $$m!! = 
\begin{cases}
  m(m-2)(m-4) \cdots 4 \cdot 2  & \text{if } m \text{ is even}, \\
  m(m-2)(m-4) \cdots 3 \cdot 1 & \text{if } m \text{ is odd}.
\end{cases}
$$
  \label{lemma:zeroth_row}
\end{lemma}
\begin{proof}
  Being rewritten as
  $$
    a_{n+1,0} = \sum\limits_{k=0}^n (-1)^{k}\binom{n}{k}\big(2(n-k)+1\big)!!,
  $$
  relation~\eqref{eq_zeroth_row_2} can be proved with the help of the inclusion-exclusion principle.
  Indeed, there are $(2n+1)!!$ matchings of size $n+1$.
  In $(2n-1)!!$ among them, $i$-th arc belongs to the pattern \ptwoone\ as an outer arc, $i=1,\ldots,n$.
  Hence, to get the number of pattern avoiding matchings, we need to deduce $\binom{n}{1}(2n-1)!!$ from $(2n+1)!!$.
  However, matchings with two occurrences of \ptwoone\ are deduced twice.
  There are $(2n-3)!!$ matchings with $i$-th and $j$-th arcs belonging to the pattern \ptwoone\ as outer arcs, and $\binom{n}{2}$ possible pairs $(i,j)$, so we add $\binom{n}{2}(2n-3)!!$.
  The same way, we treat matchings with three occurrences of \ptwoone\,, with four occurrences and so on.

\end{proof}

For any $n,k\geq0$, let us denote
$$
  b_{n,k} := a_{n+1,k}.
$$
Define also
$$
  B(z,u) :=   \sum\limits_{n=0}^{\infty}
    \sum\limits_{k=0}^n
      b_{n,k} \dfrac{z^n}{n!} u^k.
$$

\begin{lemma}
  The exponential generating function $B(z,u)$ satisfies
  \begin{equation}
    B(z,u) = \dfrac{e^{z(u-1)}}{\sqrt{(1-2z)^3}}.
    \label{eq:egf12}
  \end{equation}
    In particular, the exponential generating function of the shifted
    $k$-th row of Table~\ref{table:12} is
  $$
    [u^k]B(z,u) = \dfrac{z^k}{k!} \cdot \dfrac{e^{-z}}{\sqrt{(1-2z)^3}}.  
  $$
  \label{lemma:egf12}
\end{lemma}
\begin{proof}
  Recall that
  $$
    \sum\limits_{n=0}^{\infty} (2n+1)!! \dfrac{z^n}{n!} =
    \dfrac{1}{\sqrt{(1-2z)^3}}.
  $$
  Therefore, according to Lemma~\ref{lemma:zeroth_row}, we have
  \begin{align*}
    B(z,0)
    &=
    \sum\limits_{n=0}^{\infty} b_{n,0} \dfrac{z^n}{n!}
    =
    \sum\limits_{n=0}^{\infty} \dfrac{z^n}{n!}
      \sum\limits_{k=0}^n (-1)^{n-k}\binom{n}{k}(2k+1)!! \\
    &=
    \left(
      \sum\limits_{n=0}^{\infty} \dfrac{(-z)^n}{n!}
    \right)
    \left(
      \sum\limits_{n=0}^{\infty} (2n+1)!! \dfrac{z^n}{n!}
    \right)
    =
    \dfrac{e^{-z}}{\sqrt{(1-2z)^3}}.
  \end{align*}
  Hence, due to Lemma~\ref{l12},
  \begin{align*}
    B(z,u)
    &=
    \sum\limits_{n=0}^{\infty} \dfrac{z^n}{n!}
      \sum\limits_{k=0}^n \binom{n}{k}b_{n-k,0}u^k \\
    &=
    \left(
      \sum\limits_{n=0}^{\infty} b_{n,0} \dfrac{z^n}{n!}
    \right)
    \left(
      \sum\limits_{n=0}^{\infty} \dfrac{(zu)^n}{n!}
    \right)
    =
    \dfrac{e^{-z}}{\sqrt{(1-2z)^3}} \cdot e^{zu}.
  \end{align*}
\end{proof}

\begin{remark}
  Combinatorially, relation~\eqref{eq:egf12} represents another facet
  of the inclusion-exclusion principle.  This can be proved using the
  symbolic method, see~\cite[Section~III.7.4]{fs}.  Indeed, the
  exponential generating function $e^{zv}/\sqrt{(1-2z)^3}$ corresponds
  to the labeled product of the combinatorial class of urns
  $\mathcal{U}$ and the derivative of the combinatorial class of
  matchings $\mathcal{M}$.  Urns can be interpreted as distinguished
  arcs (marked by the variable $v$) that are inserted into a matching
  to form double arcs, that is, patterns \ptwoone\,.  Passing to the
  derivative $\mathcal{M}'$ corresponds to adding a supplementary arc
  to all matchings.  This operation guarantees that every matching
  possesses at least one arc (and hence, there is something to
  double).  Finally, the variable substitution $v=u-1$ has the meaning
  of the inclusion-exclusion itself.
\end{remark}

\subsection{Asymptotics for patterns of size 2}

Here, we provide the asymptotics of the distribution of size-2
patterns. These formulas allow us to understand in a
concise way how the pattern distribution behaves when the size of
matchings increases, without the need for precise calculations.

\begin{thm}
  The asymptotic behavior of the numbers of matchings with $k$ occurrences of pattern \ptwoone\ (resp. \ponetwo), as $n\to\infty$, is
  $$
    a_{n,k} \sim  \dfrac{1}{2^k k!}\left(\dfrac{2}{e}\right)^{n+1/2} n^n.
  $$
  \label{theorem:12_row_asymptotics}
\end{thm}
\begin{proof}
  Let us start with establishing the asymptotics of $b_{n,0}$.
  Due to Lemma~\ref{lemma:zeroth_row}, we have
  \begin{align*}
    b_{n,0}
    &=
    \sum\limits_{k=0}^n (-1)^{n-k}\binom{n}{k}(2k+1)!! \\
    &=
    (2n+1)!!\sum\limits_{k=0}^n \dfrac{(-1)^k}{k!} \cdot
      \dfrac{n(n-1)\ldots(n-k+1)}{(2n+1)(2n-1)\ldots(2n-2k+3)} \\
    &\sim
    (2n+1)!!\sum\limits_{k=0}^n  \dfrac{(-1)^k}{2^k k!}
    \sim
    \dfrac{(2n)(2n)!}{2^nn!e^{1/2}}
    \sim
    \dfrac{2^{n+3/2}}{e^{n+1/2}} \cdot n^{n+1},
  \end{align*}
  where the last passage is done with the help of Stirling's formula.
  Now, the general case can be done by Lemma~\ref{l12}:
  $$
    b_{n,k} = \binom{n}{k}b_{n-k,0} \sim
    \dfrac{n^k}{k!} \cdot \dfrac{2^{n-k+3/2}}{e^{n-k+1/2}} \cdot (n-k)^{n-k+1}
    \sim \dfrac{2^{n-k+3/2}}{k!} \cdot \dfrac{n^{n+1}}{e^{n+1/2}}.
  $$
  Passing to $a_{n,k} = b_{n-1,k}$, we get the final result.
\end{proof}

\begin{remark}
  There is an alternative proof that relies on the singularity analysis of the exact form of the series $B(z,u)$ given by Lemma~\ref{lemma:egf12}. For details of this technique, see \cite[Theorem VI.4]{fs}.
\end{remark}


\begin{cor}
  The limit distribution of pattern \ptwoone\ (resp. \ponetwo) in a uniform random matching of size $n$ follows asymptotically a Poisson law with parameter $1/2$.
  \label{cor:12_limit}
\end{cor}
\begin{proof}
  This follows from Theorem~\ref{theorem:12_row_asymptotics} and Stirling's formula:
  $$
    \dfrac{a_{n,k}}{(2n-1)!!} \sim \dfrac{e^{-1/2}}{2^kk!}.
  $$
\end{proof}

\begin{cor}
  The ratio of numbers from the $k$-th row to the numbers of the $(k+1)$-th row of Table~\ref{table:12} tends to $2(k+1)$.
  In other words, for any $k \in \N$,
  $$
    \dfrac{a_{n,k}}{a_{n,k+1}} \sim 2(k+1).
  $$
\end{cor}
\begin{proof}
  This is an immediate consequence of Corollary~\ref{cor:12_limit}.
\end{proof}

\subsection{Patterns of size 3}

In this Subsection, we examine endhered patterns of size 3.  There are
six of them and, as it follows from Lemma~\ref{lemma:etwist}, they are
divided into two equivalence classes with respect to distributions of
these patterns in matchings (see Fig.~\ref{fig:wilf}).  Thus, our goal
is to study these two different distributions represented, for
instance, by patterns 321 and 132.  For simplicity, we will use the
following notations:
$$
  c_{n,k} := a_{n,k}(321) = a_{n,k}(123)
$$
to denote the number of matchings of size $n$ containing $k$ patterns 321 (\pthreetwoone), and
$$
  d_{n,k} := a_{n,k}(132) = a_{n,k}(213) = a_{n,k}(231) = a_{n,k}(312)
$$
to designate the number of matchings of size $n$ containing $k$ patterns 132 (\ponethreetwo).

\begin{thm}
  For any $n>0$ and $k\geq0$,
  \begin{equation}
    \begin{aligned}
      c_{n,0} & = \sum_{s=0}^{\lfloor n/2 \rfloor} \binom{n-s}{s} a_{n-s,0}, \\
      c_{n,k} & = \sum_{s=1}^{\lfloor (n-k)/2 \rfloor}
      \binom{k+s-1}{k}
      \binom{n-k-s}{s}
      a_{n-k-s,0} \mbox{\quad if } k>0,
    \end{aligned}
    \label{eq:321}
  \end{equation}
  where $a_{n,0}$ is the number of matchings of size $n$ avoiding
  pattern 21 (\ptwoone).
  \label{theorem:321}
\end{thm}
\begin{proof}
  Each matchings avoiding pattern \pthreetwoone\ can be uniquely
  obtained from a matching avoiding pattern \ptwoone\ by doubling some
  of its arcs (possibly, none).  In particular, to get a matching of
  size $n$ by doubling exactly $s$ arcs, where $0\leq s\leq \lfloor n/2
  \rfloor$, we take a matching of size $n-s$ avoiding \ptwoone\,, and
  then choose $s$ of its arcs to transform them into $s$ occurrences
  of \ptwoone.  Since this can be done in $\binom{n-s}{s}$ ways, we
  obtain relation~\eqref{eq:321} for $k=0$.

  Similarly, any matching of size $n$ with $k>0$ occurrences of
  pattern \pthreetwoone\ can be uniquely obtained by the following
  procedure.  First, we fix $0< s\leq \lfloor (n-k)/2 \rfloor$ and
  take a matching of size $n-k-s$ avoiding \ptwoone\,.  Second, we
  choose $s$ arcs in this matching and double them, there are
  $\binom{n-k-s}{s}$ ways to do it.  Finally, we stack additional $k$
  arcs onto just created $s$ occurrences of \ptwoone\,, which can be
  done in $\binom{k+s-1}{k}$ ways (apply the stars-and-bars
  argument~\cite{feller}).
\end{proof}

Table~\ref{table:321} shows the first values of $c_{n,k}$.

\begin{table}[ht]
  \small
  \centering
  \begin{tabular}{r|ccccccccc|l}
    \diagbox[width=2.5em,height=2.5em]{k}{n} &  1 & 2 & 3 & 4 & 5 & 6 & 7 & 8 & 9  &  \\ \hline
0 & 1 & 3 & 14 & 100 & 906 & 10022 & 130864 & 1969884 & 33583700 & \\
1 & 0 & 0 & 1 & 4 & 34 & 332 & 3866 & 52400 & 811248 & \\
2 & 0 & 0 & 0 & 1 & 4 & 36 & 362 & 4304 & 59256 & \\
3 & 0 & 0 & 0 & 0 & 1 & 4 & 38 & 392 & 4752 & \\
4 & 0 & 0 & 0 & 0 & 0 & 1 & 4 & 40 & 422 & \\
5 & 0 & 0 & 0 & 0 & 0 & 0 & 1 & 4 & 42 & \\
6 & 0 & 0 & 0 & 0 & 0 & 0 & 0 & 1 & 4 & \\
7 & 0 & 0 & 0 & 0 & 0 & 0 & 0 & 0 & 1 & \\
  \end{tabular}
  \caption{Distribution of pattern \protect\pthreetwoone\ (321).}
\label{table:321}
\end{table}

  In order to study the distribution of the occurrences of pattern \ponethreetwo\,,
  we apply Goulden-Jackson {\em cluster method}~\cite{goulden-jackson1, goulden-jackson2}
  and the inclusion-exclusion principle in the framework of the symbolic method~\cite[Section~III.7.4]{fs}.
  Let us trace the main steps.
  
  We start by considering all possible matchings;
  the corresponding generating function is given by
  $$
    F (z) = \sum_{n=0}^\infty (2n-1)!!\, z^n =
    1 + z + 3 z^2 + 15 z^3 + 105 z^4 + \ldots
  $$
  In every matching, we distinguish certain arcs, say, by coloring them violet.
  Algebraically, this is done by passing to the generating function
  $$
    G(z,v) = \sum_{n=0}^{\infty} \sum_{k=0}^n g_{n,k} z^n v^k = F(z + zv),
  $$
  where $g_{n,k}$ is the number of matchings of size $n$ with $k$ violet arcs (marked by the variable $v$).
  Next, we replace violet arcs by ``thick arcs''
  constructed from 3 arcs forming pattern \ponethreetwo\
  (in general case, we should replace them by {\em clusters}
  consisting of intersected copies of a given pattern;
  pattern \ponethreetwo\,, however, admits no possible self-intersection).
  This corresponds to the generating function
  $$
    H(z,v) = G(z,z^2v).
  $$
  And finally, according to the inclusion-exclusion principle, we
  obtain the generating function
  $$
    D(z,u) = \sum_{n=0}^\infty \sum_{k=0}^n d_{n,k} z^n u^k = H(z,u-1)
  $$
  whose coefficients $d_{n,k}$ enumerate the matchings of size $n$ having $k$ occurrences of the pattern \ponethreetwo\
  (see Figure~\ref{fig:inclusion-exclusion} illustrating the process).

\begin{figure}[h]
  \centering
  \includegraphics[width=\textwidth]{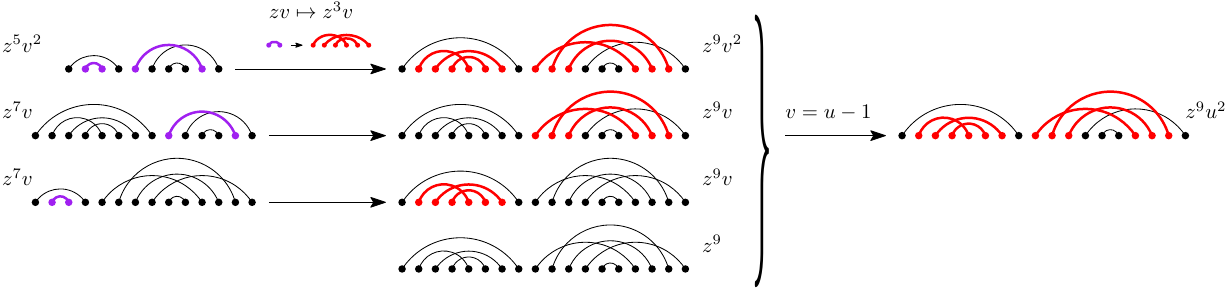}
  \caption{Cluster method and the inclusion-exclusion principle.}
  \label{fig:inclusion-exclusion}
\end{figure}

Thus, the following result takes place.

\begin{thm}
  The bivariate generating function
  $$
    D(z,u) = \sum_{n=0}^\infty \sum_{k=0}^n d_{n,k} z^n u^k,
  $$
  where $d_{n,k}$ is the number of matchings of size $n$ containing $k$ patterns \ponethreetwo\,
  can be expressed as
  $$
    D(z,u) = F\big(z + (u-1)z^3\big)  =
    \sum_{n=0}^\infty (2n-1)!!\, \big(z + (u-1)z^3\big)^n.
  $$
  \label{theorem:132}
\end{thm}

\begin{table}[ht]
  \small
  \centering
  \begin{tabular}{r|ccccccccc|l}
    \diagbox[width=2.5em,height=2.5em]{k}{n} &  1 & 2 & 3 & 4 & 5 & 6 & 7 & 8 & 9  &  \\ \hline
    0 & 1 & 3 & 14 & 99 & 900 & 9978 & 130455 & 1965285 & 33522915 & \\
    1 & 0 & 0 & 1 & 6 & 45 & 414 & 4635 & 61110 & 927090 & \\
    2 & 0 & 0 & 0 & 0 & 0 & 3 & 45 & 630 & 9405 & \\
    3 & 0 & 0 & 0 & 0 & 0 & 0 & 0 & 0 & 15 &
  \end{tabular}
  \caption{Distribution of pattern \protect\ponethreetwo\ (132).}
\label{table:132}
\end{table}

Table~\ref{table:132} shows the first values of $d_{n,k}$.

The limit distributions of patterns \pthreetwoone\ and \ponethreetwo\ are different.
In the case of \pthreetwoone\,, for a fixed $k\geq 0$, as $n\to\infty$, we have

$$
  c_{n,k} \sim \sqrt2 \cdot C_k \cdot \dfrac{(2n)^{n-k}}{e^n},
$$
where
$$
  C_0 = 1
  \qquad\mbox{and}\qquad
  C_k = \sum\limits_{s=1}^k \binom{k-1}{s-1}\dfrac{1}{2^s s!}
  \quad\mbox{if } k>0.
$$
On the other hand, for \ponethreetwo\, we have
$$
  d_{n,k} \sim \dfrac{2^{n-2k+1/2}}{k!e^n} \cdot n^{n-k},
$$
as $n\to\infty$.
In other words, for large values of $n$, we have
$$
  \dfrac{c_{n,k}}{(2n-1)!!} \sim \dfrac{C_k}{2^{k}} \cdot \dfrac{1}{n^k}
  \qquad\mbox{and}\qquad
  \dfrac{d_{n,k}}{(2n-1)!!} \sim \dfrac{1}{2^{2k}k!} \cdot \dfrac{1}{n^k},
$$

  meaning that a large uniform random matching avoids both patterns
  with high probability.  The proof of this asymptotic behavior
  requires other technical means and is beyond the scope of this
  paper.  We will present this proof as part of the analysis of the
  asymptotic behavior of any
  (more complex) endhered pattern in our next work.

\section{Occurrences of endhered patterns in RNA secondary structures and shapes}
\label{section:p}

In this section, we shift our attention to the native (real-world) data and
discuss how endhered patterns are distributed in various
representations of the secondary RNA structures obtained from Protein
Data Bank~\cite{pdb} (PDB, \url{https://www.rcsb.org}) using our Python scripts
(\url{https://gitlab.com/celiabiane/endhered_pattern}), which depends on
GEMMI~\cite{gemmi} to parse mmCIF files from PDB
(\url{https://github.com/project-gemmi/gemmi}), and
uses two methods to detect base pairs in RNA molecules:
a closed-source software
x3DNA-DSSR~\cite{3dna,dssr} and an open-source software FR3D-python~\cite{fr3d,webfr3d} (\url{https://github.com/BGSU-RNA/fr3d-python}).

  The software x3DNA-DSSR directly produces extended dot-bracket
  notations. To derive these notations x3DNA-DSSR uses only {\em
    canonical pairs}: A-U, C-G, wobble G-U, and A-T (in RNA-DNA
  hybrids) with cis. Watson-Crick/Watson-Crick interactions and
  without forming parallel mini-duplexes.  FR3D-python gives a list of
  base pairs. We parse its results, filter canonical base pairs {\em à
    la} x3DNA-DSSR, and produce extended dot-bracket notations using a
  simple First-Come-First-Served method.

  There are several hundred known, existing in nature, modifications
  of nucleotides. In the data they are denoted by special one-, two-,
  or three-letter codes, different from the classical 4 letters
  UACG. Modified nucleotides are mapped to short 1-letter nucleotide
  names.  For example A23 is mapped A, EQ0 to G, and CCC to C.
  Nucleic Acid KnowledgeBase~\cite{NAKB1, NAKB2}
  (\url{https://www.nakb.org/modifiednt.html}) contains details for
  these mappings. Following x3DNA-DSSR\footnote{See
  \url{http://forum.x3dna.org/rna-structures/modified-nucleotides-incorrect}}
  we adapt one exception to these rules: pseudouridine (PSU) is mapped
  to P and not to U.  This adaptation allows us to better compare the
  x3DNA-DSSR and FR3D-python results.
  
From the RNA secondary structures obtained with x3NA-DSSR and
FR3D-python, we collapsed unpaired nucleotides in order to keep only
paired ones. When keeping only RNA structures composed of one chain,
this leads to 1501 RNA structures, in 933 (resp. 929) of them
x3DNA-DSSR (resp. FR3D-python) has found at least one canonical base
pair. The data has been accessed on August 29, 2024.

The structures in the extended dot-bracket notation are converted to
matchings using an algorithm based on stacks. The algorithm works as
follows: the word composed of parentheses is read from left to right,
when an opening character is met its index is stacked in a stack
corresponding to the nature of the character. When a closing character
is met, a pair corresponding to the last index in the corresponding
stack and the current index is added to the matching, and the index of
the opening parenthesis is unstacked. Figure~\ref{fig:conversion}
shows an example of this conversion.

\begin{figure}[h]
    \centering
    \includegraphics[scale=0.70]{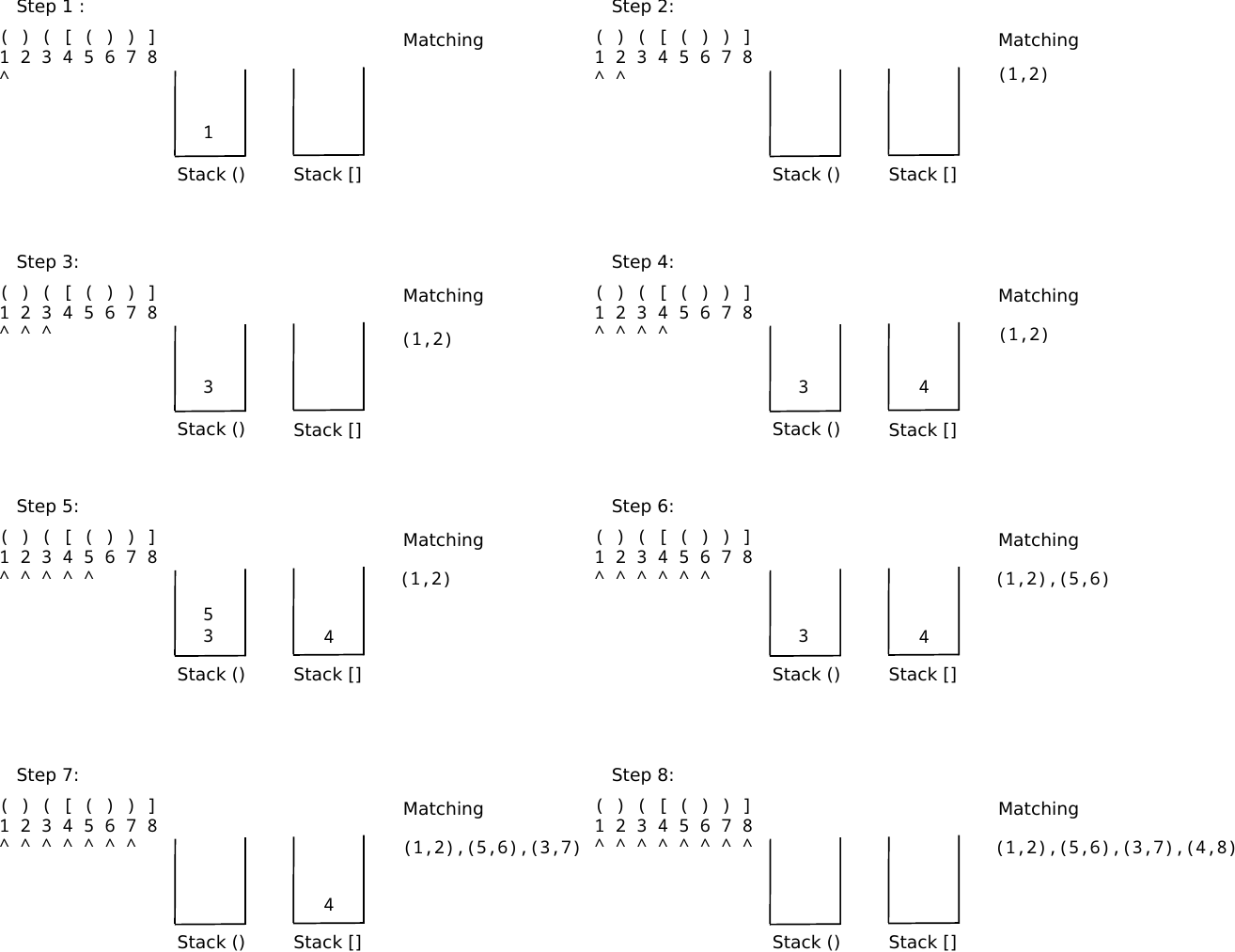}
    \caption{Example of conversion from dot-bracket notation to a matching.}
    \label{fig:conversion}
\end{figure}

Among the 933 (929 in case of FR3D-python) paired and non empty secondary structures,
926 (921 in case of FR3D-python) contains
the pattern \ptwoone. In case of x3DNA-DSSR, the seven remaining structures, with PDB
identifiers
3IZY,
5UZ9,
6B45,
6B46,
6B47,
6B48,
6VQX, contain a
single base pairing.
In case of FR3D-python, there are 12 of them:
1FC8,
1R4H,
3IZY,
4R8I,
4Z7K,
5UZ9,
6B45,
6B46,
6B47,
6B48,
6BY5,
6VQX.
A detailed comparison of the results of the two programs would be very
enriching, but is beyond the scope of this article.

None of the considered 933 structures are encoded with more than 4 types of parentheses.
In case of x3DNA-DSSR, only 11 structures
(4DS6,
4E8V,
8OLS,
8OLV,
8OLW,
8OM0,
8RUI,
8RUJ,
8RUL,
8RUM,
8RUN) are written using 4 types of parentheses, 27
employ 3 types of parentheses
(5KMZ,
5TPY,
6AGB,
6AHR,
6UES,
6WLQ,
6WLR,
6WLS,
7EZ0,
7K16,
7U4A,
7UO5,
7UVT,
7XD3,
7XD4,
7XD7,
7XSK,
7XSL,
7XSM,
7XSN,
8HD6,
8HD7,
8I7N,
8T29,
8T2A,
8T2B,
8T2O),
and 193 structures are encoded with 2 types of parentheses.
In case of FR3D-python, the corresponding numbers are 1, 39 and 184.
We need to be careful when comparing these numbers, x3DNA-DSSR
uses elimination-gain heuristics~\cite{smit}
to produce extended dot-bracket notations,
while we adapt First-Come-First-Served method.

The distribution of patterns
\ptwoone\ and \pthreetwoone\ in RNA secondary structure is shown in
Figure~\ref{fig:scatter}. We can observe that the majority of RNA
secondary structures are of small size (between 0 and 100 nucleotides),
and there seems to be a linear relation between the size of the
matching and the number of occurrence of the pattern. This can be
explained by the large number of hairpin patterns in RNA secondary
structures.

Both x3DNA-DSSR and FR3D-python detect 2 RNA structures containing the
\ponetwo\ pattern: 4M4O, 5U3G. One of them, 4M4O, corresponds to the
minE aptamer involved in a complex with a lysozyme. Both methods,
based on x3DNA-DSSR and FR3D-python, give the following secondary
structure:
$$
((((((.(((((.((.......((((.(.(.[..)]..).)))))).))))).))))))
$$

Another one, 5U3G, is the Dickeya dadantii ykkC riboswitch, which has
the following secondary structure:
$$
\begin{aligned}
\text{x3DNA-DSSR:} & \quad ((((((..((...(((((....))))).))....))))))(((((((....))))))).(.[...(((.......))).)..].. \\
\text{FR3D-python:} & \quad .(((((..((...(((((....))))).))....))))).(((((((....))))))).(.[...(((.......))).)..]..
\end{aligned}
$$
Diving into the data, we see that FR3D-python
does not detect GTP-C pair on position (1,40) in this case.

The results of two methods differ in the case of molecule 7K16 (Tamana
Bat Virus xrRNA1):
$$
\begin{aligned}
  \text{x3DNA-DSSR:} & \quad \{\{...((((((((....))))))((((...[[[...))))\}\}.))...]]]\\
  \text{FR3D-python:} & \quad  .(...[[((((((....))))))((((...\{\{\{...)))))..]]...\}\}\}
\end{aligned}
$$
In this case x3DNA-DSSR detects, in addition to FR3D-python,
a 5GP-C base pair on position (1,42). Removing this pair,
we create a \ponethreetwo\ pattern. See Table~\ref{table:data} for more details.

Pattern \ptwothreeone\ is contained only in 4M4O,
while pattern \pthreeonetwo\ belongs only
to 5U3G. Pattern \pthreetwoone\ is contained in 918 (908 in case of FR3D-python)
structures.

It is surprising that pattern \ponetwo\ appears so rarely in RNA
structures while pseudoknots are thought to have important biological
functions.  We observe that occurrence of this pattern are ``hidden''
by the high frequency of nested bonds (see
Figure~\ref{fig:pat12-nested-bonds}). To neutralize this effect, we
pass to RNA shapes.

\begin{figure}[h]
\centering
  \includegraphics{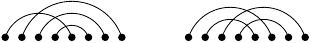}
\caption{Two examples of matchings with nesting bonds (occurrences of \protect\ptwoone)  but without occurrences of \protect\ponetwo.}
\label{fig:pat12-nested-bonds}
\end{figure}

\begin{table}[ht]
  \centering
  \scriptsize
  \begin{tabularx}{\textwidth}{c|p{4em}|p{4em}|X|X}
    Pattern   & \multicolumn{2}{c|}{RNA secondary structure} & \multicolumn{2}{c}{Extended RNA shape} \\ \hline
              & x3DNA-DSSR & FR3D-python & x3DNA-DSSR & FR3D-python  \\ \hline
    \ptwoone\ &  926 & 921 & 0 & 0 \\\hline
    \ponetwo\ & 2 (4M4O, 5U3G) & 3 (4M4O, 5U3G, 8SH5) &
    59
    (1A60,
    1E95,
    1KAJ,
    1KPD,
    1KPY,
    1KPZ,
    1L2X,
    1L3D,
    1RNK,
    1YG3,
    1YG4,
    1YMO,
    2A43,
    2AP0,
    2AP5,
    2G1W,
    2K95,
    2K96,
    2LC8,
    2M58,
    2M8K,
    2RP0,
    2RP1,
    2TPK,
    2XDB,
    3IYQ,
    3IYR,
    3IZ4,
    3U4M,
    3U56,
    3UMY,
    437D,
    4M4O,
    4PQV,
    4QG3,
    4QVI,
{\bf    4R8I,}
    5KMZ,
    5NPM,
    5TPY,
    5U3G,
    6AGB,
    6AHR,
    6DLQ,
    6DLR,
    6DLS,
    6DLT,
    6DNR,
    6E1T,
    6E1V,
    6VUH,
    7K16,
    7U4A,
    7UO5,
    8HIO,
    8T29,
    8T2A,
    8T2B,
    8T2O)
    &
    60
    (1A60,
    1E95,
    1KAJ,
    1KPD,
    1KPY,
    1KPZ,
    1L2X,
    1L3D,
    1RNK,
    1YG3,
    1YG4,
    1YMO,
    2A43,
    2AP0,
    2AP5,
    2G1W,
    2K95,
    2K96,
    2LC8,
    2M58,
    2M8K,
    2RP0,
    2RP1,
    2TPK,
    2XDB,
    3IYQ,
    3IYR,
    3IZ4,
    3U4M,
    3U56,
    3UMY,
    437D,
    4M4O,
    4PQV,
    4QG3,
    4QVI,
    5KMZ,
    5NPM,
    5TPY,
    5U3G,
    6AGB,
    6AHR,
{\bf    6D3P,}
    6DLQ,
    6DLR,
    6DLS,
    6DLT,
    6DNR,
    6E1T,
    6E1V,
    6VUH,
    7K16,
    7U4A,
    7UO5,
    8HIO,
{\bf    8SH5,}
    8T29,
    8T2A,
    8T2B,
    8T2O)
    \\\hline
    \ptwothreeone\ & 1 (4M4O) & 1 (4M4O) & 8 (3U4M, 3U56, 3UMY, 4M4O, 4QG3, 4QVI, {\bf 4R8I}, 5NPM)
                                         & 7 (3U4M, 3U56, 3UMY, 4M4O, 4QG3, 4QVI, 5NPM)\\\hline
    \pthreeonetwo\ & 1 (5U3G) & 1 (5U3G) & 13 (3U4M, 3U56, 3UMY, 4QG3, 4QVI, 5KMZ, 5NPM, 5U3G, 6DLQ, 6DLR, 6DLS, 6DLT, 6DNR)

                                         & 13 (3U4M, 3U56, 3UMY, 4QG3, 4QVI, 5KMZ, 5NPM, 5U3G, 6DLQ, 6DLR, 6DLS, 6DLT, 6DNR) \\\hline
    \ponethreetwo\ & 0        & 1 (7K16) & 0 & 0 \\\hline
    \pthreetwoone\ & 918      & 908 & 0 & 0 \\\hline
    \ptwoonethree\, \ponetwothree & 0 & 0 & 0 & 0\\ \hline
  \end{tabularx}
  \caption{
    RNA secondary structures and RNA shapes in which a given pattern occurs at least once.
    Differences are highlighted in bold.}
  \label{table:data}
\end{table}
 
The concept of RNA shapes have been introduced by Giegerich, Voss, and
Rehmsmeier in 2004~\cite{giegerich}.  In these shapes, no unpaired
regions are included and nested bonds are combined. For instance,
the secondary structure $$..(((.((..(((....))).(((.....))))))))..$$
has the following RNA shape:
$$(()())$$
Originally, the RNA shapes have
been defined in words with a single type of parenthesis. They are
counted by Motzkin numbers~\cite{choi,giegerich,delest-viennot} and
correspond exactly to non-crossing matchings avoiding the endhered
pattern \ptwoone\,.

We adapt the Giegerich-Voss-Rehmsmeier reduction to matchings with
crossings represented by words with different types of
parentheses. The result of this adaptation is what Rødland called {\em
  collapsed structures}~\cite{rodland}.  This is done by keeping only
$(i,j)$ pairs in matching such that $(i+1,j-1)$ does not belong to the
matching and then reindexing the pairs. The number of
\ptwoone\ patterns in resulting reduced matchings (RNA shapes) is
obviously $0$. Interestingly, the number of RNA shapes with at least
one occurrence of pattern \ponetwo\ increases up to 59 (60 with
FR3D-python). Among those, 8 (7 with FR3D-python) RNA structures have
pattern \ptwothreeone\ and 13 have pattern \pthreeonetwo\,. Other size
3 patterns (\ponethreetwo,\pthreetwoone,\ptwoonethree,\ponetwothree)
are not detected in RNA shapes.  It is expected for patterns
\ponethreetwo, \pthreetwoone, and \ptwoonethree, as they contain the
pattern \ptwoone\ forbidden in RNA shapes.

\begin{figure}[h]
 \centering
 \includegraphics[width=1\textwidth]{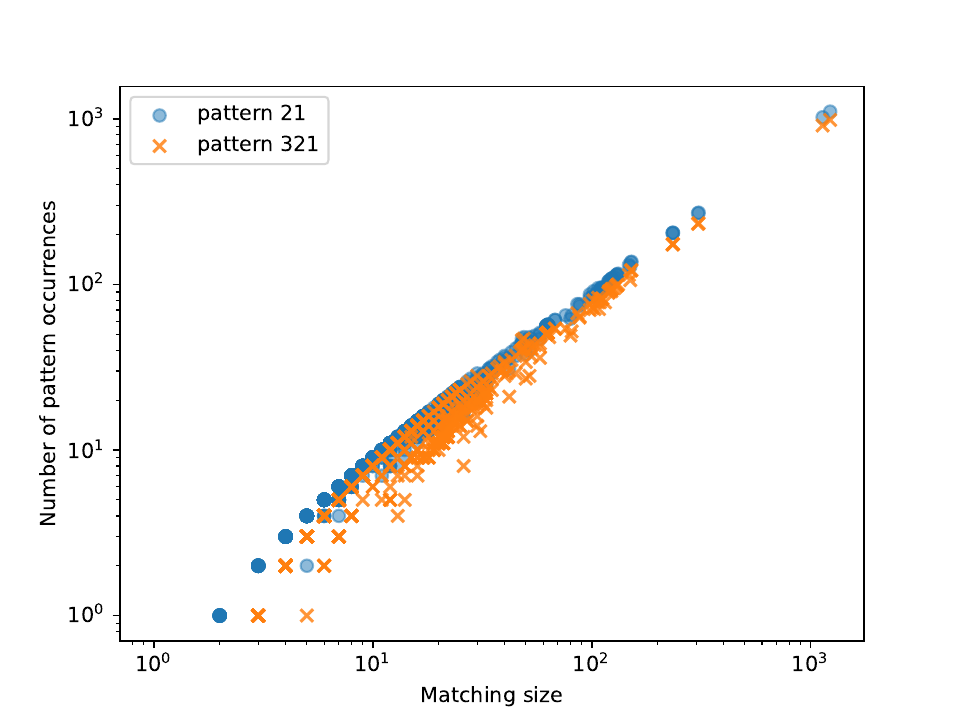}
 \caption{Scatter plot of the number of occurrences of patterns 21 and
   321 as a function of the size of the matching. Axes are in
   logarithmic scale.  RNA with no occurrences of patterns 21 and 321
   are not displayed. Matchings are obtained using FR3D-python, the
   results for x3DNA-DSSR are similar and available on git
   repository \url{https://gitlab.com/celiabiane/endhered_pattern}.  }
    \label{fig:scatter}
\end{figure}

In this paper, we study only endhered patterns of size 2 and 3.  The
analysis of more complex patterns in native (real-world) data may be a
part of future works.  The function {\tt count\_pattern} in {\tt
count\_visualisation.py} from our GitLab repository can be applied
to any endhered pattern.

\section{Conclusion and discussions}
\label{section:c}

We have examined distributions of endhered patterns of sizes 2 and 3
in matchings from theoretico-combinatorial and data-driven points of
view.  In matchings, patterns \ponetwo\ and \ptwoone\ have the same
distribution. Six endhered patterns of size 3 are divided into 2
equidistributed groups: \ponetwothree\,, \pthreetwoone\ and
\ponethreetwo\,, \pthreeonetwo\,, \ptwoonethree\,, \ptwothreeone\,.
Moreover, the joint distribution of patterns of the same size is
symmetric if the patterns are equivalent under the twist operation.
We have also provided corresponding asymptotic behavior of these distributions.

 In this work, we deliberately abstract from the nucleic acid sequences,
 and model the secondary structures directly by matchings.  Our
 results show that there is a big difference between observed and
 modelled pattern distributions. This means that non-restricted
 matchings are too permissive and new models should be developed to get closer to the observed pattern distributions.  We
 wonder if it is possible to describe essential features of RNA
 secondary structures with pseudoknots, using pattern-based
 restrictions. The classical Waterman's definition of RNA structures
 and its generalizations (see Subsection~\ref{subsec:models}) can be
 regarded as an example of pattern-based restriction used, among other
 things, to control the prediction algorithms for secondary structures
 from nucleic acid sequences. More insights into patterns frequencies
 in the native secondary RNA structures may guide us towards
 mixture of two definitions, complex enough to cover different
 pseudoknot-like structures presented in real data, but at the same
 time quite simple and neat to prevent the uncontrolled combinatorial
 explosion.

 The non-existence of certain short sequences in genomic
   and protein sequences is a well-known fact~\cite{nullomers1, her,
     maw0, maw1}.  Applications include cancer research~\cite{cgt,
     9s1r} and forensic science~\cite{gos}. Less is known about the
   forbidden secondary structures, although some interesting works
   about theoretical (im)possibility of inverse RNA folding have been
   published recently~\cite{uRNA, uRNA-ponty}.  One of the following
   research directions would be to determine what influence the
   distribution of endhered patterns in the native RNA.  Some
   configurations are probably forbidden due to physicochemical
   constraints on the bending of RNA (something like Waterman-Ponty
   restrictions, but for the case that include pseudoknots), others may
   not be present because of biological reasons.  Are there some
   evolutionary mechanisms that divert the distribution of patterns
   from the theoretically observed in the equi-probabilistic model of
   matchings? How pattern distributions in secondary structures are
   related to RNA dynamics and function?
 
 For any new combinatorial characterization of RNA structures, we need
 to develop a method for estimating their affinity with structures
 observed in native molecules.  One such method could also be
 pattern-based: compare the distribution of patterns in native RNA
 with theoretically calculated distributions over matchings avoiding
 certain patterns. Our exploratory study suggest, for instance, that
 the patterns \ptwoonethree\ and \ponetwothree\ never appear in
 RNA. Moreover, PDB references presented in Table~\ref{table:data}
 look very interesting, especially 4M4O, 5U3G, and 7K16.
 
\section*{Acknowledgments}

We would like to express our immeasurable gratitude to Matteo
Cervetti, Yann Ponty for the discussions they had with us during the
beginning of the work on the endhered patterns, to Justin Masson for
help with python code, to Daniel Pinson for proofreading the article,
and to the anonymous reviewers for their valuable comments and
suggestions.  This research was funded, in part, by the Agence
Nationale de la Recherche (ANR), grant ANR-22-CE48-0002. For the
purpose of Open Access, in accordance with the grant's open access
conditions, a CC-BY public copyright licence has been applied by the
authors to the present document and will be applied to all subsequent
versions up to the Author Accepted Manuscript arising from this
submission. Versions of the present document are available here:
\url{https://arxiv.org/abs/2404.18802},
\url{https://hal.science/hal-04563757}, and
\url{https://kirgizov.link/publications/endhered-first.pdf}.
Authors were supported, in part, by
Bourgogne-Franche-Comté region (France), project ANER ARTICO.

\bibliographystyle{plain} 
\bibliography{biblio}

\end{document}